\documentclass[11pt,a4paper]{article}
\usepackage{amsfonts}
\usepackage{amsmath}
\usepackage{amssymb}
\usepackage{amsthm}
\usepackage{enumerate}
\numberwithin{equation}{section} \topmargin -1cm \oddsidemargin 2mm
\begin{document}
\title{\textbf{The Quaternionic Cauchy--Szeg\"{o} Kernel on the Quaternionic Siegel
Half Space}}
\author{Jinxun Wang\thanks{\scriptsize Department of
Mathematics, Faculty of Science and Technology, University of Macau,
Taipa, Macao, China. E-mail: wangjx08@163.com}\,, ~Xingmin
Li\thanks{\scriptsize School of Computer Sciences, South China
Normal University, Guangzhou 510631, China. E-mail: lxmin57@163.com}
~and Jianquan Liao\thanks{\scriptsize Department of Mathematics,
Guangdong Institute of Education, Guangzhou 510303, China. E-mail:
liaojq80@163.com}}
\date{}
\maketitle

\parbox{0.9 \hsize}{\textbf{Abstract.} The quaternionic Cauchy--Szeg\"{o}
kernel of the Hardy space $\mathcal{H}^2(\mathcal{U}_n)$ on the
quaternionic Siegel half space $\mathcal{U}_n$ is derived and the
Hardy spaces on the octonionic Siegel half space is
investigated.}\vskip 0.3cm
\parbox{0.9 \hsize}{\textbf{Keywords:} quaternion, octonion, Heisenberg group, Hardy
space, Cauchy--Szeg\"{o} kernel}\vskip 0.3cm
\parbox{0.9 \hsize}{\textbf{AMS (2010) Subject Classification:} 30G35, 31B05, 31B10,
42B30}

\section{Introduction and main results}
The Siegel upper half space in $\mathbb{C}^{n+1}$ is defined by
$$\mathcal {U}^n=\Big\{z\in \mathbb{C}^{n+1}:~\mbox{Im}\,z_{n+1}>\sum_
{j=1}^n|z_j|^2\Big\},$$ the space $\mathcal {H}^2(\mathcal {U}^n)$,
consists of all functions $F$ holomorphic on $\mathcal{U}^n$, for
which
$$\sup_{\varepsilon>0}\int_{\partial\mathcal{U}^n}|F_\varepsilon(z)|^2
d\beta(z)<\infty,$$ where
$$F_\varepsilon(z)=F(z+\varepsilon\mathbf{i}),\quad
\mathbf{i}=(\underbrace{0,\ldots,0}
_{\scriptsize{(2n+1)\,0\mbox{s}}},i).$$

For any $F\in\mathcal {H}^2(\mathcal {U}^n),$ the boundary limit
$F^b$ exists in the sense of almost everywhere and
$L^2(\partial\mathcal{U}^n).$ Moreover, we have the integral
representation (\cite{G, Ste})
$$
F(z)=\int_{\partial\mathcal{U}^n}S(z,\omega)
F^b(\omega)d\beta(\omega),\quad z\in \mathcal{U}^n,
$$
where
$$S(z,\omega)=\frac{n!}{4\pi^{n+1}}\frac{1}{r^{n+1}(z,\omega)}$$
is called the complex Cauchy--Szeg\"{o} kernel and
$r(z,\omega)=\frac{i}{2} (\overline{\omega_{n+1}}-z_{n+1})-\sum_
1^nz_k\overline{\omega_k}.$

It is natural to ask that whether one could built up an analogous
theory in quaternionic analysis, or even in octonionic analysis? As
we know, the Heisenberg group of several complex variables can be
identified with the boundary of the Siegel upper half space
$\mathcal{U}^n$ (\cite{Ste}), this suggests that we should first
investigate the quaternionic Heisenberg group. To our knowledge, the
quaternionic Heisenberg group was first introduced by Barker and
Salamon in their paper \cite{BS}, and rediscovered in \cite{Ch},
where the Hardy space $\mathcal{H}^2(\mathcal{U}_1)$ of two
quaternionic variables was studied. But alas, the Cauchy--Szeg\"{o}
kernel obtained in \cite{Ch} is not correct. In this paper, we
derive the desired quaternionic Cauchy--Szeg\"{o} kernel of the
quaternionic Hardy space $\mathcal{H}^2(\mathcal{U}_n)$ on the
quaternionic Siegel half space
$\mathcal{U}_n\subset\mathbb{H}^{n+1}$ of several quaternionic
variables. The difficulty lies in how to determine the explicit form
of the kernel and how to deal with the involved higher order partial
derivatives of the quaternionic Cauchy kernel. These obstacles arise
because the structure of left (right) $\mathbb{H}$-regular functions
is much more complicated than that of holomorphic functions in
complex analysis. For instance, $z^n$ is holomorphic for all $n\in
\mathbb{Z}$, but there is no integer $n$ such that $q^n$ is
$\mathbb{H}$-regular, except for $n=0$. After that, we prove the
$L^p$ boundedness of the Cauchy--Szeg\"{o} projection operator, and
give a characterization for the Hardy space consists of octonionic
analytic functions on the octonionic Siegel half space, which
generalizes the quaternionic Hardy space
$\mathcal{H}^p(\mathcal{U}_n)$ to the context of octonions.

The quaternionic Siegel half space in $\mathbb{H}^{n+1}$ ($n\geq1$)
is defined by
$$\mathcal{U}_n=\big\{(q',q_{n+1})\in \mathbb{H}^{n+1}:~
\mbox{Re}\,q_{n+1}>|q'|^2\big\},$$ where $q'=(q_1,q_2,\ldots,q_n)$
and $|q'|=(\sum_1^n|q_i|^2)^\frac{1}{2}$. To simplify the notations
we write
$\overline{q'}=(\overline{q_1},\overline{q_2},\ldots,\overline{q_n})$,
$\omega'+ q'=\sum_1^n(\omega_i+q_i)$ and $\omega'\cdot
q'=\sum_1^n\omega_iq_i.$

For any $\delta>0,$ define the dilation on $\mathcal{U}_n$ by
$\delta\circ(q',q_{n+1})=(\delta q',\delta^2q_{n+1}),$ and for any
$$
\mathcal{R}=(\mathcal{R}_1,\mathcal{R}_2,\ldots,\mathcal{R}_n)\in
\mathbb{H}^n~(|\mathcal{R}_i|=1,i=1,\ldots,n),
$$
define the rotation on $\mathcal{U}_n$ by
$\mathcal {R}(q',q_{n+1})=(\mathcal {R}\cdot q',q_{n+1}).$

We denote by $\mathcal{Q}_n$ the quaternionic Heisenberg group,
viz., $\mathcal{Q}_n=\mathbb{H}^n\times
\mathbb{R}^3=\big\{[\omega',t]:~\omega' \in
\mathbb{H}^n,\,t=(t_1,t_2,t_3) \in \mathbb{R}^3\big\},$ and define
the multiplication by
$[\alpha,t]\diamond[\beta,s]=[\alpha+\beta,t_1+s_1-2\mbox{Im}_1\,
(\overline{\beta}\cdot\alpha),t_2+s_2-2\mbox{Im}_2\,
(\overline{\beta}\cdot\alpha),t_3+s_3-2\mbox{Im}_3\,
(\overline{\beta}\cdot\alpha)].$

For each $h=[\omega',t] \in \mathcal{Q}_n,$ define the translation
on $\mathcal{U}_n$ by
$$q=(q',q_{n+1})\mapsto h(q)=(q'+\omega',q_{n+1}+|\omega'|^2+2
\overline{\omega'}\cdot q'+\mathbf{e}\cdot t),$$ here
$\mathbf{e}\cdot t=\sum_{i=1}^{3}e_it_i$, and $e_i$ ($1\leq i\leq
3$) are three imaginary units of $\mathbb{H}$ (see Section 2).

The three automorphisms of the domain $\mathcal{U}_n$ mentioned
above contribute a lot to the calculation of the Cauchy--Szeg\"{o}
kernel.

For $0<p<\infty$, the Hardy space $\mathcal{H}^p(\mathcal{U}_n)$
consists of all functions $F(q)$ which are left $\mathbb{H}$-regular
on $\mathcal{U}_n$ and satisfy
$$\|F\|_{\mathcal{H}^p(\mathcal{U}_n)}:=\Big(\sup_{\varepsilon>0}
\int_{\partial\mathcal{U}_n}
|F_\varepsilon(q)|^pd\beta(q)\Big)^{1/p}<\infty,$$ where
$$F_\varepsilon(q)=F(q+\varepsilon\mathbf{e_0}),\quad\mathbf{e_0}=
(\underbrace{0,\ldots,0}_{\scriptsize{4n\,0\mbox{s}}},1,0,0,0),$$
and $d\beta(q)$ can be identified with the Lebesgue measure on
$\partial\mathcal{U}_n$.

With these notations, we obtain the following result:
\newtheorem{theo}{Theorem}[section]
\begin{theo}\label{cs}
Suppose $F\in \mathcal{H}^p(\mathcal{U}_n)$
$(\frac{2}{3}<p<\infty)$, then
\begin{enumerate}[(i)]
\item There exists an $F^b\in L^p(\partial\mathcal{U}_n),$ such that
$F_\varepsilon(q)|_{\partial\mathcal{U}_n}\rightarrow F^b$
$(\varepsilon\rightarrow 0)$ in the sense of
$L^p(\partial\mathcal{U}_n)$ norm and almost everywhere.

\item $\{F^b\}$ is a closed subspace of
 $L^p(\partial\mathcal {U}_n),$ and $\|F^b\|_{L^p(\partial\mathcal
{U}_n)}=\|F\|_{\mathcal{H}^p(\mathcal{U}_n)}.$

\item If $F\in \mathcal{H}^2(\mathcal{U}_n),$ then
$$F(q)=\int_{\partial\mathcal{U}_n}S(q,\omega)F^b
(\omega)d\beta(\omega), \quad q=(q',q_{n+1})\in \mathcal{U}_n,$$ where
$$S(q,\omega)=s(q_{n+1}+\overline{\omega_{n+1}}
-2\overline{\omega'}\cdot q')$$ is called the quaternionic
Cauchy--Szeg\"{o} kernel and
$$s(\nu)=\Big(\frac{2}{\pi}\Big)^{2n}\frac{\partial^{2n}}{\partial
x_0^{2n}}E(\nu),$$
$$E(\nu)=\frac{1}{\sigma_4}\frac{\overline{\nu}}{|\nu|^4}=\frac{1}
{2\pi^2}\frac{\overline{\nu}} {|\nu|^4},
\quad\nu=\sum_{i=0}^{3}x_ie_i\in \mathbb{H}.$$
\end{enumerate}
\end{theo}

If we use $\mathbb{A}$ to denote the complex field $\mathbb{C}$ or
the skew field $\mathbb{H},$ $m=2$ or $4$ is the dimension of
$\mathbb{A}$ over $\mathbb{R},$ then the Siegel half space in
$\mathbb{A}^{n+1}$ can be written uniformly as
$\mathcal{D}_n=\big\{(\zeta',\zeta_{n+1})\in
\mathbb{A}^{n+1}:\mbox{Re}\,\zeta_{n+1}>|\zeta'|^2\big\}.$
Similarly, we can define the function spaces
$\mathcal{H}^p(\mathcal{D}_n),$ thus we can write the complex-valued
and the quaternion-valued Cauchy--Szeg\"{o} kernels in one form:

\begin{theo}
For every $F\in \mathcal{H}^2(\mathcal{D}_n),$ we can represent $F$
by the integral formula
$$F(\zeta)=\int_{\partial\mathcal{D}_n}S(\zeta,\omega)F^b
(\omega)d\beta(\omega),\quad \zeta=(\zeta',\zeta_{n+1})\in
\mathcal{D}_n,$$ where
$$
S(\zeta,\omega)=s(\zeta_{n+1}+\overline{\omega_{n+1}}
-2\overline{\omega'}\cdot \zeta'),
$$
$$s(\nu)=\Big(-\frac{2}{\pi}\Big)^{\frac{mn}{2}}
\frac{\partial^{\frac{mn}{2}}}{\partial x_0^{\frac{mn}{2}}}E(\nu),$$
$$E(\nu)=\frac{1}{\sigma_m}\frac{\overline{\nu}}
{|\nu|^m}=\frac{\Gamma(\frac{m}{2})} {2\pi^\frac{m}{2}}
\frac{\overline{\nu}}{|\nu|^m},\quad\nu=\sum_ix_ie_i\in
\mathbb{A}.$$
\end{theo}

The rest of this paper is organized as follows: Section 2 contains
some basic knowledge of quaternion algebra, Cayley algebra and
respectively, their analysis. Section 3 is mainly devoted to the
proof of Theorem \ref{cs}. In the last section, we discuss the
octonionic Siegel half space $\mathcal{U}$, the octonionic
Heisenberg group $\mathcal{O}$ and the octonionic Hardy spaces
$\mathcal{H}^p(\mathcal{U})$.

\section{Preliminaries}

If an algebra $\mathbb{A}$ is meanwhile a normed vector space, and
its norm ``$\|\cdot\|$'' satisfies $\|ab\|=\|a\|\|b\|$, then we call
$\mathbb{A}$ a normed algebra. If $ab=0$ ($a, b\in \mathbb{A}$)
implies $a=0$ or $b=0$, then we call $\mathbb{A}$ a division
algebra. Early in 1898, Hurwitz had proved that the real numbers
$\mathbb{R},$ complex numbers $\mathbb{C},$ quaternions $\mathbb{H}$
and octonions $\mathbb{O}$ are the only normed division algebras
over $\mathbb{R}$ (\cite{Hur}), with the imbedding relation
$\mathbb{R}\subseteq \mathbb{C}\subseteq \mathbb{H}\subseteq
\mathbb{O}$.

Any quaternion $q\in \mathbb{H}$ is of the form
$q=\sum_{i=0}^3x_ie_i$ (we often omit the algebraic unit element
$e_0$ and write $q=x_0+x_1e_1+x_2e_2+x_3e_3$), where $\mbox{Re}\,q=
x_0$, $\mbox{Im}_i\,q= x_i$ ($1\leq i\leq 3$) all belong to
$\mathbb{R}$ and the basis $e_i$ ($0\leq i\leq 3$) satisfies
$e_0^2=e_0,$ $e_ie_0=e_0e_i=e_i,$ $e_i^2=-1$ ($1\leq i\leq 3$) and
$$e_1e_2=e_3=-e_2e_1,~e_2e_3=e_1=-e_3e_2,~e_3e_1=e_2=-e_1e_3.$$
$|q|=(\sum_0^3x_i^2)^\frac{1}{2}$ is the norm of $q,$ and
$\overline{q}=x_0e_0-\sum_{i=1}^3x_ie_i$ is the conjugate of $q$.
The inverse of $q$ ($q\neq0$) is given by
$q^{-1}=\overline{q}/|q|^2$, due to
$q\overline{q}=\overline{q}q=|q|^2$. For any $q_1,q_2,q_3\in
\mathbb{H},$ we have $|q_1q_2|=|q_1||q_2|,$
$\overline{q_1q_2}=\overline{q_2} \, \overline{q_1}$ and
$(q_1q_2)q_3=q_1(q_2q_3).$ With respect to the multiplication law,
quaternions $\mathbb{H}$ is associative but non-commutative.

A function $f\in C^1(\Omega,\mathbb{H})$ is said to be left (right)
$\mathbb{H}$-regular in the open set $\Omega\subset \mathbb{R}^4$ if
and only if $$Df=e_0\frac{\partial f}{\partial x_0}+e_1
\frac{\partial f}{\partial x_1}+e_2\frac{\partial f}{\partial
x_2}+e_3\frac{\partial f}{\partial x_3}=0$$
$$\Big(fD=\frac{\partial f}{\partial x_0}e_0+
\frac{\partial f}{\partial x_1}e_1+\frac{\partial f}{\partial
x_2}e_2+\frac{\partial f}{\partial x_3}e_3=0 \Big).$$ If a left
$\mathbb{H}$-regular function is meanwhile right
$\mathbb{H}$-regular, then we call it an $\mathbb{H}$-regular
function.

The study of quaternionic analysis was started from 1930s
(\cite{F1,F2,Su}), now it becomes more and more important in physics
and engineering.

As the largest normed division algebra, octonion is discovered by
John T. Graves in 1843, and then by Arthur Cayley in 1845
independently, which is sometimes referred to as Cayley number or
the Cayley algebra, it is an 8 dimensional division algebra over
$\mathbb{R}$ with the basis $e_0,e_1,\ldots,e_7$ satisfying
$e_0^2=e_0$, $e_ie_0=e_0e_i=e_i,$ $e_i^2=-1,$ $i=1,2,\ldots,7.$
Denote
$$W=\{(1,2,3),(1,4,5),(1,7,6),(2,4,6),(2,5,7),(3,4,7),(3,6,5)\},$$
then any triple $(\alpha,\beta,\gamma)\in W$ obeys $$e_\alpha
e_\beta=e_\gamma=-e_\beta e_\alpha,\quad e_\beta
e_\gamma=e_\alpha=-e_\gamma e_\beta,\quad e_\gamma
e_\alpha=e_\beta=- e_\alpha e_\gamma,$$ which completely determine
the multiplication of octonions by linearity.

For any $x=\sum_0^7 x_ie_i \in \mathbb{O}$, $\mbox{Re}\,x=x_{0}$ is
called the scalar (real) part of $x$ and
$\overrightarrow{x}=\sum_1^7 x_ie_i$ is called its vector part. The
$i$th component $x_i$ is denoted by $\mbox{Im}_i\,x$ ($1\leq i\leq
7).$ Define
$\overline{x}=\sum_0^7x_i\overline{e_i}=x_0-\overrightarrow{x}$ as
the conjugate of $x$, and $|x|=(\sum_0^7x_i^2)^\frac{1}{2}$ as the
norm (modulus) of $x$, they satisfy: $|xy|=|x||y|,$
$x\overline{x}=\overline{x}x=|x|^2,$
$\overline{xy}=\overline{y}\,\overline{x}$ $(x,y\in \mathbb{O}).$ If
$x\neq0,$ $x^{-1}=\overline{x}/{|x|^2}$ gives the inverse of $x.$

Octonionic multiplication is neither commutative nor associative.
$[x, y, z]=(xy)z-x(yz)$ is called the associator of $x, y, z\in
\mathbb{O},$ which satisfies
$$
[x,y,z]=[y,z,x]=-[y,x,z], \quad [x,x,y]=0=[\overline{x},x,y].
$$

The octonionic analysis, which is a generalization of quaternionic
analysis to higher dimensions, was studied systematically since 1995
(\cite{L1}). Suppose $\Omega$ is an open subset of $\mathbb{R}^8$,
if $f\in C^1(\Omega,\mathbb{O})$ satisfies
$Df=\sum_0^{7}e_{i}\frac{\partial f}{\partial x_{i}}=0$
($fD=\sum_0^{7} \frac{\partial f}{\partial x_{i}}e_{i}=0$), then $f$
is said to be left (right) $\mathbb{O}$-analytic in $\Omega$, here
the Dirac operator $D$ and its conjugate $\overline{D}$ are defined
by $D=\sum_0^7e_i\frac{\partial}{\partial x_i}$ and $\overline
D=\sum_0^7\overline{e_i}\frac{\partial}{\partial x_i}$ respectively.
A function $f$ is $\mathbb{O}$-analytic means that $f$ is left
$\mathbb{O}$-analytic, and also right $\mathbb{O}$-analytic. From
$\overline{D}(Df)=(\overline{D}D)f=\triangle
f=f(D\overline{D})=(fD)\overline{D}$, we know that any left (right)
$\mathbb{O}$-analytic ($\mathbb{H}$-regular) function is always
harmonic.

In order to build up the theory of $H^p$-spaces in
higher-dimensional Euclidean spaces, in 1960, E. M. Stein and G.
Weiss generalized the notion of holomorphic functions to the system
of conjugate harmonic functions (\cite{SW}), which is now called the
Stein--Weiss conjugate harmonic system, it is a vector of harmonic
functions $(\mu_0, \mu_1, \ldots, \mu_n)$ of variables
$(x_0,x_1,\ldots,x_n)$, whose components satisfy the following
generalized Cauchy--Riemann equations:

\begin{equation*}
\left\{
\begin{aligned}
&\sum_{i=0}^n\frac{\partial\mu_i}{\partial x_i}=0,\\
&\frac{\partial\mu_j}{\partial x_k}=\frac{\partial \mu_k}{\partial
x_j},\quad 0\leq j<k\leq n.
\end{aligned} \right.
\end{equation*}

If $F(x_0,x_1,\ldots,x_7)=(f_0,f_1,\ldots,f_7)$ is a Stein--Weiss
conjugate harmonic system in $\Omega$, then
$\overline{F}=f_0-\sum_{i=1}^7f_ie_i$ is an $\mathbb{O}$-analytic
function (\cite{L1}). But inversely, this is not true (\cite{LL}).
For more information and recent progress about octonionic analysis,
we refer the reader to [2, 12--17].

\section{The proof of Theorem \ref{cs} and the Cauchy--Szeg\"{o} projection operator}
\begin{proof}[\textbf{Proof of Theorem \ref{cs}}]
By analogous discussions as in \cite{Ste}, we can easily prove (i)
and (ii). As to (iii), by the same method in the case of several
complex variables, we can show the existence and uniqueness of
$S(q,\omega),$ and $S(q,\omega)=\overline{S(\omega,q)}.$ The proofs
are omitted here.

Invoking the three automorphisms of the domain $\mathcal{U}_n$
mentioned in Section 1, we can derive the following integral
identities:
\begin{align*}
&F(q)=\int_{\partial\mathcal{U}_n}S(\delta\circ
q,\delta\circ\omega)\delta^{4n+6}F^b(\omega)d\beta(\omega),&
&\forall \delta>0,\\
&F(q)=\int_{\partial\mathcal{U}_n}S(\mathcal{R}(q),\mathcal
{R}(\omega))F^b(\omega)d\beta(\omega),& &\forall \mathcal {R},\\
&F(q)=\int_{\partial\mathcal{U}_n}S(h(q),h(\omega))F^b(\omega )
d\beta(\omega),& &\forall h\in\mathcal {Q}_n.
\end{align*}
These imply that $$S(q,\omega)=S(\delta\circ q,\delta\circ\omega)
\delta^{4n+6}=S(\mathcal {R}(q),\mathcal
{R}(\omega))=S(h(q),h(\omega)).$$

Hence, $s(q_{n+1}):=S(q,0)$ is independent of $q'$ and satisfies
$$S(q,\omega)=s(q_{n+1}+\overline{\omega_{n+1}}-2\overline
{\omega'}\cdot q').$$ Further more, $s(q_{n+1})$ is left $\mathbb
{H}$-regular with respect to one quaternionic variable $q_{n+1}$ and
homogeneous of negative order $-2n-3.$ Similar to the method in
\cite{BDS}, we can show that
$$\Big\{\frac{\partial^{2n}
}{\partial x_0^r\partial x_1^s\partial x_2^t}E(\nu):
r+s+t=2n\Big\}$$ forms a basis of this homogeneous class, where
$E(\nu)=\frac{1}{2\pi^2}\frac{\overline{\nu}}{|\nu|^4}\,
(\nu=\sum_{i=0}^{3}x_ie_i\in \mathbb{H})$ is called the quaternionic
Cauchy kernel which is $\mathbb{H}$-regular in
$\mathbb{R}^4\setminus\{0\}$. So the function $s(\cdot)$ must be of
the form
$$s(\nu)=\sum_{r+s+t=2n}\frac{\partial^{2n}
E(\nu)}{\partial x_0^r\partial x_1^s\partial x_2^t}c_{r,s,t}.$$ The
rest of the proof will be devoted to the determination of the
coefficients $c_{r,s,t}\in \mathbb{H}$.

For this purpose, it is convenient for the readers that we give an
outline here. That is, first, we build up a linear system of
infinite equations that contains the undetermined coefficients by
the special value method, and then solve the system. Here we should
point out that in the case of several complex variables, to
determine the complex Cauchy--Szeg\"{o} kernel, we just need to
compute one unknown coefficient, and the computation is direct and
easy. While in the quaternionic case, there are $(n+1)(2n+1)$
unknown coefficients to compute, due to the dimension of the space
of homogeneous left $\mathbb{H}$-regular functions of negative order
$-(2n+3)$ is $C_{2n+2}^2=(n+1)(2n+1)$. Also, in such case, the
higher order partial derivatives of the quaternionic Cauchy kernel
$E(\nu)$ are not easy to handle with, unless $n$ is specific and
small. So we would use some different techniques to deal with it. To
be specific, we calculate the multi-integrals by Fourier transform
and spherical coordinates transform. While solving the system, we
adopt the limiting arguments.

In what follows we use $N$ to denote the Newton potential
$\frac{1}{|\nu|^2}=\frac{1}{x_0^2+x_1^2+x_2^2+x_3^2}$ in
$\mathbb{R}^4$. We start by proving two basic properties concerning
the multi-integrals.

\newtheorem{prop}{Proposition}[section]
\begin{prop}
\label{p1}For any $p_0+p_1+p_2+p_3=\alpha$ and
$q_0+q_1+q_2+q_3=\gamma$, we have
\begin{align*}
\int_{\mathbb{R}^3}&\frac{\partial^{\alpha}N}{\partial x_0^{p_0}
\partial x_1^{p_1}\partial
x_2^{p_2}\partial x_3^{p_3}}\frac{\partial^{\gamma}N}{\partial
x_0^{q_0}\partial x_1^{q_1}\partial x_2^{q_2}\partial
x_3^{q_3}}dx_1dx_2dx_3\\=
\int_{\mathbb{R}^3}&2^{\alpha+\gamma}\pi^{\alpha+\gamma+2}
(-1)^{p_0+\gamma}i^{\alpha+\gamma-p_0-q_0}
(x_1^2+x_2^2+x_3^2)^{\frac{p_0+q_0}{2}-1}\\ &\times
x_1^{p_1+q_1}x_2^{p_2+q_2}x_3^{p_3+q_3}e^{-4\pi x_0
\sqrt{x_1^2+x_2^2+x_3^2}}dx_1dx_2dx_3.
\end{align*}
\end{prop}
\renewcommand\proofname{Proof}
\begin{proof}
By Parseval's theorem,
\begin{align*}
&\int_{\mathbb{R}^3}\frac{\partial^{\alpha}N}{\partial
x_0^{p_0}\partial x_1^{p_1}\partial x_2^{p_2}\partial
x_3^{p_3}}\frac{\partial^{\gamma}N}{\partial x_0^{q_0}
\partial x_1^{q_1}\partial x_2^{q_2}\partial
x_3^{q_3}}dx_1dx_2dx_3
\\=&\int_{\mathbb{R}^3}\Big(\frac{\partial^{\alpha}N}{\partial
x_0^{p_0}\partial x_1^{p_1}\partial x_2^{p_2}\partial
x_3^{p_3}}\Big)^{\widehat{}}(x_1,x_2,x_3)
\overline{\Big(\frac{\partial^{\gamma}N}{\partial x_0^{q_0}\partial
x_1^{q_1}\partial x_2^{q_2}\partial x_3^{q_3}}
\Big)^{\widehat{}}(x_1,x_2,x_3)}dx_1dx_2dx_3\\=&\int_{\mathbb{R}^3}(2\pi
i)^{p_1+p_2+p_3}x_1^{p_1}x_2^{p_2}x_3^{p_3}\frac{\partial^{p_0}
\widehat{N}}{\partial x_0^{p_0}}(-2\pi
i)^{q_1+q_2+q_3}x_1^{q_1}x_2^{q_2}x_3^{q_3}
\overline{\frac{\partial^{q_0}\widehat{N}}{\partial
x_0^{q_0}}}dx_1dx_2dx_3,
\end{align*} and
\begin{align*}
\widehat{N}(x_1,x_2,x_3)&=\int_{\mathbb{R}^3}\frac{e^{-2\pi
i(x_1\xi_1+x_2\xi_2+x_3\xi_3)}}{x_0^2+\xi_1^2+\xi_2^2+ \xi_3^2}d\xi
\\&=\int_{\mathbb{R}^3}\frac{e^{-2\pi i\xi_1
\sqrt{x_1^2+x_2^2+x_3^2}}}{x_0^2+\xi_1^2+\xi_2^2+ \xi_3^2}d\xi
\\&=\int_{\mathbb{R}^2}\Big(\int_\mathbb{R}\frac{\cos(2\pi
\xi_1\sqrt{x_1^2+x_2^2+x_3^2})}{x_0^2+
\xi_1^2+\xi_2^2+\xi_3^2}d\xi_1\Big)d\xi_2d\xi_3
\\&=\int_{\mathbb{R}^2}\frac{\pi}{\sqrt{x_0^2+\xi_2^2+\xi_3^2}}e^{-2\pi
\sqrt{(x_1^2+x_2^2+x_3^2)(x_0^2+\xi_2^2+\xi_3^2)}}d\xi_2d\xi_3
\\&=\frac{\pi}{\sqrt{x_1^2+x_2^2+x_3^2}}e^{-2\pi
x_0\sqrt{x_1^2+x_2^2+x_3^2}},
\end{align*}
the proposition follows.
\end{proof}
\begin{prop}
\label{p2}Suppose $a>0$, denote $l=l_0+l_1+l_2+l_3$, we have
\begin{align*}
&\int_{\mathbb{R}^3}(x_1^2+x_2^2+x_3^2)^{\frac{l_0}{2}}
x_1^{l_1}x_2^{l_2}x_3^{l_3}
e^{-a\sqrt{x_1^2+x_2^2+x_3^2}}dx_1dx_2dx_3
\\=&\left\{\begin{aligned}
&2a^{-l-3}\Gamma(l+3)\frac{\Gamma(k_1+\frac{1}{2})
\Gamma(k_2+\frac{1}{2})\Gamma(k_3+\frac{1}{2})}
{\Gamma(k_1+k_2+k_3+\frac{3}{2})},&
&\mbox{if}~l_1=2k_1,l_2=2k_2,l_3=2k_3,\\
&0,& &\mbox{else}.
\end{aligned}\right.
\end{align*}
\end{prop}
\begin{proof}
Invoking the spherical coordinates, when $l_1=2k_1$, $l_2=2k_2$,
$l_3=2k_3$, we have
\begin{align*}
&\int_{R^3}(x_1^2+x_2^2+x_3^2)^{\frac{l_0}{2}}
x_1^{l_1}x_2^{l_2}x_3^{l_3}
e^{-a\sqrt{x_1^2+x_2^2+x_3^2}}dx_1dx_2dx_3
\\=&\int_0^{\infty}r^{l+2}e^{-ar}dr\int_0^\pi
{\cos}^{2k_1}\theta{\sin}^{2k_2+2k_3+1}\theta d\theta\int_0^{2\pi}
{\cos}^{2k_2}\varphi {\sin}^{2k_3}\varphi
d\varphi\\=&a^{-l-3}\Gamma(l+3)
\frac{\Gamma(k_1+\frac{1}{2})\Gamma(k_2+k_3+1)}
{\Gamma(k_1+k_2+k_3+\frac{3}{2})}\frac{2\Gamma(k_2+\frac{1}{2})
\Gamma(k_3+\frac{1}{2})}{\Gamma(k_2+k_3+1)}\\=&
2a^{-l-3}\Gamma(l+3)\frac{\Gamma(k_1+\frac{1}{2})
\Gamma(k_2+\frac{1}{2})\Gamma(k_3+\frac{1}{2})}
{\Gamma(k_1+k_2+k_3+\frac{3}{2})}.
\end{align*}

Otherwise, if, for example, $l_1$ is odd, then the integrand is odd
with respect to the variable $x_1$, so the integral vanishes.
\end{proof}

Now we turn to the calculation of the coefficients. First we note
that
$$E(\nu)=-\frac{1}{4\pi^2}\overline{D}N=
-\frac{1}{4\pi^2}\Big(\frac{\partial N}{\partial x_0}-\frac{\partial
N} {\partial x_1}e_1- \frac{\partial N}{\partial
x_2}e_2-\frac{\partial N}{\partial x_3}e_3\Big),$$ we may rewrite
$s(\nu)$ as
\begin{align*}
s(\nu)=\sum_{s_0+s_1+s_2=2n}\Big(&\frac{\partial^{2n+1}N} {\partial
x_0^{s_0+1}\partial x_1^{s_1}\partial
x_2^{s_2}}-\frac{\partial^{2n+1}N}{\partial x_0^{s_0}\partial
x_1^{s_1+1}\partial x_2^{s_2}}e_1-\frac{\partial^{2n+1}N}{\partial
x_0^{s_0}\partial x_1^{s_1}\partial x_2^{s_2+1}}e_2\\
&-\frac{\partial^{2n+1}N}{\partial x_0^{s_0}
\partial x_1^{s_1}\partial x_2^{s_2}\partial
x_3}e_3\Big)c(s_0,s_1,s_2),
\end{align*}
where
$$c(s_0,s_1,s_2)=c_0(s_0,s_1,s_2)+c_1(s_0,s_1,s_2)e_1+c_2
(s_0,s_1,s_2)e_2+c_3(s_0,s_1,s_2)e_3$$ are to be
determined. Set
\begin{align*}
F_\lambda(q',q_{n+1})=\Big(&\frac{\partial^{\lambda+1}N}{\partial
x_0^{t_0+1}\partial x_1^{t_1}\partial x_2^{t_2}\partial
x_3^{t_3}}-\frac{\partial^{\lambda+1}N}{\partial x_0^{t_0}\partial
x_1^{t_1+1}\partial x_2^{t_2}\partial x_3^{t_3}}e_1\\
&-\frac{\partial^{\lambda+1}N}{\partial x_0^{t_0}\partial
x_1^{t_1}\partial x_2^{t_2+1}\partial
x_3^{t_3}}e_2-\frac{\partial^{\lambda+1}N}{\partial
x_0^{t_0}\partial x_1^{t_1}\partial x_2^{t_2}\partial
x_3^{t_3+1}}e_3\Big)\Big|_{\nu=1+q_{n+1}},
\end{align*}
where $\lambda=t_0+t_1+t_2+t_3$, then one can verify that
$F_\lambda(q',q_{n+1})\in \mathcal {H}^2(\mathcal{U}_n)$ whenever
$\lambda>(2n-3)/2$. By the integral representation, we have

$$
F_\lambda(0,1)=\int_{\partial\mathcal{U}_n}
S(q,\omega)|_{q=(0,1)}F_\lambda^b(\omega)d\beta(\omega)
=\int_{\partial\mathcal{U}_n}\overline{s(1+\omega_{n+1})}
F_\lambda(\omega)d\beta(\omega)
$$
$$
=\sum_{s_0+s_1+s_2=2n}\big(c_0(s_0,s_1,s_2)
-c_1(s_0,s_1,s_2)e_1-c_2(s_0,s_1,s_2)e_2 -c_3(s_0,s_1,s_2)e_3\big)
$$
$$\times\bigg(\int_{\mbox{\scriptsize{Re}}\,\omega_{n+1}=|\omega'|^2}
\Big(\frac{\partial^{2n+1}N}{\partial x_0^{s_0+1}\partial
x_1^{s_1}\partial x_2^{s_2}}\frac{\partial^{\lambda+1}N}{\partial
x_0^{t_0+1}\partial x_1^{t_1}\partial x_2^{t_2}\partial x_3^{t_3}}+
\frac{\partial^{2n+1}N}{\partial x_0^{s_0}\partial
x_1^{s_1+1}\partial x_2^{s_2}}\frac{\partial^{\lambda+1}N}{\partial
x_0^{t_0}\partial x_1^{t_1+1}\partial x_2^{t_2}\partial x_3^{t_3}}$$
$$+\frac{\partial^{2n+1}N}{\partial x_0^{s_0}\partial
x_1^{s_1}\partial x_2^{s_2+1}}\frac{\partial^{\lambda+1}N}{\partial
x_0^{t_0}\partial x_1^{t_1}\partial x_2^{t_2+1}\partial
x_3^{t_3}}+\frac{\partial^{2n+1}N}{\partial x_0^{s_0}\partial
x_1^{s_1}\partial x_2^{s_2}\partial
x_3}\frac{\partial^{\lambda+1}N}{\partial x_0^{t_0}\partial
x_1^{t_1}\partial x_2^{t_2}\partial
x_3^{t_3+1}}\Big)\Big|_{\nu=1+\omega_{n+1}}d\beta(\omega)$$

$$+\int_{\mbox{\scriptsize{Re}}\,\omega_{n+1}=
|\omega'|^2}\Big(-\frac{\partial^{2n+1}N}{\partial x_0^{s_0+1}
\partial x_1^{s_1}\partial
x_2^{s_2}}\frac{\partial^{\lambda+1}N}{\partial x_0^{t_0}\partial
x_1^{t_1+1}\partial x_2^{t_2}\partial
x_3^{t_3}}+\frac{\partial^{2n+1}N}{\partial x_0^{s_0}\partial
x_1^{s_1+1}\partial x_2^{s_2}}\frac{\partial^{\lambda+1}N}{\partial
x_0^{t_0+1}\partial x_1^{t_1}\partial x_2^{t_2}\partial x_3^{t_3}}$$
$$+\frac{\partial^{2n+1}N}{\partial x_0^{s_0}\partial
x_1^{s_1}\partial x_2^{s_2}\partial x_3}
\frac{\partial^{\lambda+1}N}{\partial x_0^{t_0}\partial
x_1^{t_1}\partial x_2^{t_2+1}\partial
x_3^{t_3}}-\frac{\partial^{2n+1}N}{\partial x_0^{s_0}\partial
x_1^{s_1}\partial x_2^{s_2+1}}\frac{\partial^{\lambda+1}N}{\partial
x_0^{t_0}\partial x_1^{t_1}\partial x_2^{t_2}\partial
x_3^{t_3+1}}\Big)\Big|_{\nu=1+\omega_{n+1}}d\beta(\omega)e_1$$

$$+\int_{\mbox{\scriptsize{Re}}\,\omega_{n+1}=
|\omega'|^2}\Big(-\frac{\partial^{2n+1}N}{\partial x_0^{s_0+1}
\partial x_1^{s_1}\partial
x_2^{s_2}}\frac{\partial^{\lambda+1}N}{\partial x_0^{t_0}\partial
x_1^{t_1}\partial x_2^{t_2+1}\partial
x_3^{t_3}}+\frac{\partial^{2n+1}N}{\partial x_0^{s_0}\partial
x_1^{s_1}\partial x_2^{s_2+1}}\frac{\partial^{\lambda+1}N}{\partial
x_0^{t_0+1}\partial x_1^{t_1}\partial x_2^{t_2}\partial x_3^{t_3}}$$
$$+\frac{\partial^{2n+1}N}{\partial x_0^{s_0}\partial
x_1^{s_1+1}\partial x_2^{s_2}} \frac{\partial^{\lambda+1}N}{\partial
x_0^{t_0}\partial x_1^{t_1}\partial x_2^{t_2}\partial x_3^{t_3+1}}
-\frac{\partial^{2n+1}N}{\partial x_0^{s_0}\partial
x_1^{s_1}\partial x_2^{s_2}\partial
x_3}\frac{\partial^{\lambda+1}N}{\partial x_0^{t_0}\partial
x_1^{t_1+1}\partial x_2^{t_2}\partial
x_3^{t_3}}\Big)\Big|_{\nu=1+\omega_{n+1}}d\beta(\omega)e_2$$

$$+\int_{\mbox{\scriptsize{Re}}\,\omega_{n+1}=
|\omega'|^2}\Big(-\frac{\partial^{2n+1}N}{\partial x_0^{s_0+1}
\partial x_1^{s_1}\partial
x_2^{s_2}}\frac{\partial^{\lambda+1}N}{\partial x_0^{t_0}\partial
x_1^{t_1}\partial x_2^{t_2}\partial
x_3^{t_3+1}}+\frac{\partial^{2n+1}N}{\partial x_0^{s_0}\partial
x_1^{s_1}\partial x_2^{s_2}\partial
x_3}\frac{\partial^{\lambda+1}N}{\partial x_0^{t_0+1}\partial
x_1^{t_1}\partial x_2^{t_2}\partial x_3^{t_3}}$$
$$+\frac{\partial^{2n+1}N}{\partial x_0^{s_0}
\partial x_1^{s_1}\partial x_2^{s_2+1}}
\frac{\partial^{\lambda+1}N}{\partial x_0^{t_0}\partial
x_1^{t_1+1}\partial x_2^{t_2}\partial
x_3^{t_3}}-\frac{\partial^{2n+1}N}{\partial x_0^{s_0}\partial
x_1^{s_1+1}\partial x_2^{s_2}}\frac{\partial^{\lambda+1}N}{\partial
x_0^{t_0}\partial x_1^{t_1}\partial x_2^{t_2+1}\partial
x_3^{t_3}}\Big)\Big|_{\nu=1+\omega_{n+1}}d\beta(\omega)e_3\bigg).$$
By Proposition \ref{p1},
\begin{align}
F_\lambda(0,1)=\sum_{s_0+s_1+s_2=2n}\big(&c_0(s_0,s_1,s_2)
-c_1(s_0,s_1,s_2)e_1-c_2(s_0,s_1,s_2)e_2
-c_3(s_0,s_1,s_2)e_3\big)\nonumber
\\
\times\bigg(&\int_{\mathbb{H}^n}\Big(\int_{\mathbb{R}^3}2^{2n+\lambda+3}
\pi^{2n+\lambda+4}(-1)^{s_0+\lambda}i^{2n+\lambda-s_0-t_0}
(x_1^2+x_2^2+x_3^2)^{\frac{s_0+t_0}{2}}\nonumber
\\&
\times x_1^{s_1+t_1}x_2^{s_2+t_2}x_3^{t_3}e^{-4\pi
x_0\sqrt{x_1^2+x_2^2+x_3^2}}dx_1dx_2dx_3\Big)
\Big|_{x_0=1+|\omega'|^2}d\omega'\nonumber
\\+&\int_{\mathbb{H}^n}\Big(\int_{\mathbb{R}^3}2^{2n+\lambda+3}\pi^{2n+\lambda+4}
(-1)^{s_0+\lambda+1}i^{2n+\lambda+1-s_0-t_0}
(x_1^2+x_2^2+x_3^2)^{\frac{s_0+t_0-1}{2}}\nonumber
\\& \times x_1^{s_1+t_1+1}x_2^{s_2+t_2}x_3^{t_3}e^{-4\pi
x_0\sqrt{x_1^2+x_2^2+x_3^2}}dx_1dx_2dx_3\Big)
\Big|_{x_0=1+|\omega'|^2}d\omega'e_1\nonumber
\\+&\int_{\mathbb{H}^n}\Big(\int_{\mathbb{R}^3}2^{2n+\lambda+3}\pi^{2n+\lambda+4}
(-1)^{s_0+\lambda+1}i^{2n+\lambda+1-s_0-t_0}
(x_1^2+x_2^2+x_3^2)^{\frac{s_0+t_0-1}{2}}\nonumber
\\& \times x_1^{s_1+t_1}x_2^{s_2+t_2+1}x_3^{t_3}e^{-4\pi
x_0\sqrt{x_1^2+x_2^2+x_3^2}}dx_1dx_2dx_3\Big)
\Big|_{x_0=1+|\omega'|^2}d\omega'e_2\nonumber
\\+&\int_{\mathbb{H}^n}\Big(\int_{\mathbb{R}^3}2^{2n+\lambda+3}\pi^{2n+\lambda+4}
(-1)^{s_0+\lambda+1}i^{2n+\lambda+1-s_0-t_0}
(x_1^2+x_2^2+x_3^2)^{\frac{s_0+t_0-1}{2}}\nonumber
\\&
\times x_1^{s_1+t_1}x_2^{s_2+t_2}x_3^{t_3+1}e^{-4\pi
x_0\sqrt{x_1^2+x_2^2+x_3^2}}dx_1dx_2dx_3\Big)\Big|_{x_0=1+
|\omega'|^2}d\omega'e_3\bigg).\label{sys}
\end{align}

When $t_i$ ($0\leq i\leq 3$) vary, we will get a system of linear
equations in variables $c_i(s_0,s_1,s_2)$ ($0\leq i\leq 3,
s_0+s_1+s_2=2n$). Now, we solve this system.

Taking $t_1=2q_1$, $t_2=2q_2$, $t_3=2q_3+1$, one can prove by
induction that
\begin{align*}
F_\lambda(0,1)&=(-1)^{t_0+q_1+q_2+q_3}
2^{-\lambda}\frac{\Gamma(\lambda+3)\Gamma(q_1+q_2+q_3+2)\Gamma(2q_1)
\Gamma(2q_2)\Gamma(2q_3+2)}{\Gamma(2q_1+2q_2+2q_3+4)
\Gamma(q_1)\Gamma(q_2)\Gamma(q_3+1)}e_3\\&=(-1)^{t_0+q_1+q_2+q_3}
2^{-\lambda-4}\pi^{-1}\frac{\Gamma(\lambda+3)
\Gamma(q_1+\frac{1}{2})\Gamma(q_2+\frac{1}{2})\Gamma(q_3+\frac{3}{2})}
{\Gamma(q_1+q_2+q_3+\frac{5}{2})}e_3,
\end{align*}
here the formula $\Gamma(x)\Gamma(x+\frac{1}{2})=2^{1-2x}
\Gamma(\frac{1}{2})\Gamma(2x)$ is used. Applying Proposition
\ref{p2} to the right hand side of system (\ref{sys}), we get
\begin{align}
&\sum_{p_0+p_1+p_2=n}\big(c_0(2p_0,2p_1,2p_2)
-c_1(2p_0,2p_1,2p_2)e_1-c_2(2p_0,2p_1,2p_2)e_2
-c_3(2p_0,2p_1,2p_2)e_3\big)\nonumber\\&\times
(-1)^{t_0+q_1+q_2+q_3+n+p_0+1}2^{-2n-\lambda-2}\pi^{2n+1}
\frac{\Gamma(\lambda+3)
\Gamma(p_1+q_1+\frac{1}{2})\Gamma(p_2+q_2+\frac{1}{2})
\Gamma(q_3+\frac{3}{2})} {\Gamma(q_1+q_2+q_3+n-p_0+\frac{5}{2})}e_3
\nonumber\\=&F_\lambda(0,1)=(-1)^{t_0+q_1+q_2+q_3}
2^{-\lambda-4}\pi^{-1}\frac{\Gamma(\lambda+3)
\Gamma(q_1+\frac{1}{2})\Gamma(q_2+\frac{1}{2})\Gamma(q_3+\frac{3}{2})}
{\Gamma(q_1+q_2+q_3+\frac{5}{2})}e_3.
\end{align}
Hence,
\begin{align}\label{eq1}
\sum_{p_0+p_1+p_2=n}(-1)^{p_0}\frac{\Gamma(p_1+q_1+\frac{1}{2})
\Gamma(p_2+q_2+\frac{1}{2})}
{\Gamma(q_1+q_2+q_3+n-p_0+\frac{5}{2})}c_i(2p_0,2p_1,2p_2)=0,~1\leq
i\leq 3,
\end{align} and
\begin{align}\label{eq2}
\sum_{p_0+p_1+p_2=n}(-1)^{n+p_0+1} \frac{\Gamma(p_1+q_1+\frac{1}{2})
\Gamma(p_2+q_2+\frac{1}{2})}{\Gamma(q_1+q_2+q_3+n-p_0+\frac{5}{2})}
c_0(2p_0,2p_1,2p_2)=
\frac{2^{2n-2}}{\pi^{2n+2}}\frac{\Gamma(q_1+\frac{1}{2})
\Gamma(q_2+\frac{1}{2})}{\Gamma(q_1+q_2+q_3+\frac{5}{2})}.
\end{align}

While taking $t_1=2q_1+1$, $t_2=2q_2$, $t_3=2q_3+1$, we will get
\begin{align}
&\sum_{p_0+p_1+p_2=n-1}\overline{c(2p_0+1,2p_1+1,2p_2)}
(-1)^{t_0+q_1+q_2+q_3+n+p_0+1}2^{-2n-\lambda-2}\pi^{2n+1}\nonumber\\
&\times\frac{\Gamma(\lambda+3)\Gamma(p_1+q_1+\frac{3}{2})
\Gamma(p_2+q_2+\frac{1}{2})\Gamma(q_3+\frac{3}{2})}
{\Gamma(q_1+q_2+q_3+n-p_0+\frac{5}{2})}e_3
=F_\lambda(0,1)=0.
\end{align}
Hence,
\begin{align}\label{eq3}
\sum_{p_0+p_1+p_2=n-1}(-1)^{p_0}\frac{\Gamma(p_1+q_1+\frac{3}{2})
\Gamma(p_2+q_2+\frac{1}{2})}
{\Gamma(q_1+q_2+q_3+n-p_0+\frac{5}{2})}c_i(2p_0+1,2p_1+1,2p_2)=0,~0\leq
i\leq 3.
\end{align}

Similarly, taking $t_1=2q_1$, $t_2=2q_2+1$, $t_3=2q_3+1$, we get
\begin{align}
&\sum_{p_0+p_1+p_2=n-1}\overline{c(2p_0+1,2p_1,2p_2+1)}
(-1)^{t_0+q_1+q_2+q_3+n+p_0+1}2^{-2n-\lambda-2}\pi^{2n+1}\nonumber\\
&\times\frac{\Gamma(\lambda+3)\Gamma(p_1+q_1+\frac{1}{2})
\Gamma(p_2+q_2+\frac{3}{2})\Gamma(q_3+\frac{3}{2})}
{\Gamma(q_1+q_2+q_3+n-p_0+\frac{5}{2})}e_3 =F_\lambda(0,1)=0,
\end{align}
which gives
\begin{align}\label{eq4}
\sum_{p_0+p_1+p_2=n-1}(-1)^{p_0}\frac{\Gamma(p_1+q_1+\frac{1}{2})
\Gamma(p_2+q_2+\frac{3}{2})}
{\Gamma(q_1+q_2+q_3+n-p_0+\frac{5}{2})}c_i(2p_0+1,2p_1,2p_2+1)=0,~0\leq
i\leq 3.
\end{align}
And taking $t_1=2q_1+1$, $t_2=2q_2+1$, $t_3=2q_3+1$, we get
\begin{align}
&\sum_{p_0+p_1+p_2=n-1}\overline{c(2p_0,2p_1+1,2p_2+1)}
(-1)^{t_0+q_1+q_2+q_3+n+p_0}2^{-2n-\lambda-2}\pi^{2n+1}\nonumber\\
&\times\frac{\Gamma(\lambda+3)\Gamma(p_1+q_1+\frac{3}{2})
\Gamma(p_2+q_2+\frac{3}{2})\Gamma(q_3+\frac{3}{2})}
{\Gamma(q_1+q_2+q_3+n-p_0+\frac{7}{2})}e_3
=F_\lambda(0,1)=0,
\end{align} which gives
\begin{align}\label{eq5}
\sum_{p_0+p_1+p_2=n-1}(-1)^{p_0}\frac{\Gamma(p_1+q_1+\frac{3}{2})
\Gamma(p_2+q_2+\frac{3}{2})}
{\Gamma(q_1+q_2+q_3+n-p_0+\frac{7}{2})}c_i(2p_0,2p_1+1,2p_2+1)=0,~0\leq
i\leq 3.
\end{align}

Now we solve the equations (\ref{eq1}), (\ref{eq2}), (\ref{eq3}),
(\ref{eq4}) and (\ref{eq5}). We would like to deal with (\ref{eq2})
first. Multiplying (\ref{eq2}) by $\Gamma(q_1+q_2+q_3+\frac{5}{2})$
at both sides, one gets
\begin{align}
&\sum_{p_1+p_2=n}\frac{(-1)^{n+1}\Gamma(p_1+q_1+\frac{1}{2})
\Gamma(p_2+q_2+\frac{1}{2})c_0(0,2p_1,2p_2)}
{(q_1+q_2+q_3+n-1+\frac{5}{2})(q_1+q_2+q_3+n-2+\frac{5}{2})
\cdots(q_1+q_2+q_3+\frac{5}{2})}\nonumber\\
+&\sum_{p_1+p_2=n-1}\frac{(-1)^n\Gamma(p_1+q_1+\frac{1}{2})
\Gamma(p_2+q_2+\frac{1}{2})c_0(2,2p_1,2p_2)}
{(q_1+q_2+q_3+n-2+\frac{5}{2})(q_1+q_2+q_3+n-3+\frac{5}{2})
\cdots(q_1+q_2+q_3+\frac{5}{2})}\nonumber\\
+&\cdots+\sum_{p_1+p_2=1}\frac{\Gamma(p_1+q_1+\frac{1}{2})
\Gamma(p_2+q_2+\frac{1}{2})}
{q_1+q_2+q_3+\frac{5}{2}}c_0(2n-2,2p_1,2p_2)\nonumber\\
+&(-1)\Gamma(q_1+\frac{1}{2})\Gamma(q_2+\frac{1}{2})c_0(2n,0,0)=
\frac{2^{2n-2}}{\pi^{2n+2}}\Gamma(q_1+\frac{1}{2})
\Gamma(q_2+\frac{1}{2}).\label{eq6}
\end{align}
Fix $q_1$ and $q_2$, let $q_3\rightarrow\infty$, we immediately get
\begin{equation}\label{an1}
c_0(2n,0,0)=-2^{2n-2}/\pi^{2n+2}.
\end{equation}
Hence (\ref{eq6}) can be reduced to
\begin{align}
&\sum_{p_1+p_2=n}\frac{(-1)^{n+1}\Gamma(p_1+q_1+\frac{1}{2})
\Gamma(p_2+q_2+\frac{1}{2})c_0(0,2p_1,2p_2)}
{(q_1+q_2+q_3+n-1+\frac{5}{2})(q_1+q_2+q_3+n-2+\frac{5}{2})
\cdots(q_1+q_2+q_3+\frac{5}{2})}\nonumber\\
+&\sum_{p_1+p_2=n-1}\frac{(-1)^n\Gamma(p_1+q_1+\frac{1}{2})
\Gamma(p_2+q_2+\frac{1}{2})c_0(2,2p_1,2p_2)}
{(q_1+q_2+q_3+n-2+\frac{5}{2})(q_1+q_2+q_3+n-3+\frac{5}{2})
\cdots(q_1+q_2+q_3+\frac{5}{2})}\nonumber\\
+&\cdots+\sum_{p_1+p_2=1}\frac{\Gamma(p_1+q_1+\frac{1}{2})
\Gamma(p_2+q_2+\frac{1}{2})}
{q_1+q_2+q_3+\frac{5}{2}}c_0(2n-2,2p_1,2p_2)=0.\label{eq7}
\end{align}
Multiplying (\ref{eq7}) by $q_1+q_2+q_3+\frac{5}{2}$ and letting
$q_3\rightarrow\infty$, we get
\begin{equation}
\label{eq8}\Gamma(q_1+\frac{1}{2})
\Gamma(q_2+\frac{3}{2})c_0(2n-2,0,2)+
\Gamma(q_1+\frac{3}{2})\Gamma(q_2+\frac{1}{2})c_0(2n-2,2,0)=0.
\end{equation}
Dividing (\ref{eq8}) by $\Gamma(q_1+\frac{3}{2})$, we obtain
\begin{equation}\label{eq9}
\frac{\Gamma(q_2+\frac{3}{2})}{q_1+\frac{1}{2}}c_0(2n-2,0,2)+
\Gamma(q_2+\frac{1}{2})c_0(2n-2,2,0)=0.
\end{equation}
Now fix $q_2$, let $q_1\rightarrow\infty$, we immediately get
\begin{equation}
\label{an2}c_0(2n-2,2,0)=0,
\end{equation}
and consequently,
\begin{equation}
\label{an3}c_0(2n-2,0,2)=0.
\end{equation}
Repeating the above process, after (\ref{an3}) we will successively
obtain
\begin{equation}\label{an4}
c_0(2p_0,2p_1,2p_2)=0,\quad\forall~p_0\leq n-2,~p_0+p_1+p_2=n.
\end{equation}

Through similar discussions on (\ref{eq1}), (\ref{eq3}), (\ref{eq4})
and (\ref{eq5}), we can show that
\begin{align}
c_i(2p_0,2p_1,2p_2)&=0,\quad\forall~1\leq i\leq
3,~p_0+p_1+p_2=n,\label{an5}\\
c_i(2p_0+1,2p_1+1,2p_2)&=0,\quad\forall~0\leq i\leq
3,~p_0+p_1+p_2=n-1,\label{an6}\\
c_i(2p_0+1,2p_1,2p_2+1)&=0,\quad\forall~0\leq i\leq
3,~p_0+p_1+p_2=n-1,\label{an7}\\
c_i(2p_0,2p_1+1,2p_2+1)&=0,\quad\forall~0\leq i\leq
3,~p_0+p_1+p_2=n-1.\label{an8}
\end{align}

Combining (\ref{an1}), (\ref{an2})--(\ref{an8}), we have in fact
proved that
$$c(2n,0,0)=-2^{2n-2}/\pi^{2n+2},\quad
c(s_0,s_1,s_2)=0~(\mbox{other cases}).$$ So
$$s(\nu)=-2^{2n-2}/\pi^{2n+2}\frac{\partial^{2n}}{\partial
x_0^{2n}}(\overline{D}N)=(2/\pi)^{2n}\frac{\partial^{2n}}{\partial
x_0^{2n}}E(\nu).$$ The proof of Theorem \ref{cs} is complete.
\end{proof}

Now we can define the Cauchy--Szeg\"{o} projection operator
$\mathcal{C}$ in terms of the Cauchy--Szeg\"{o} kernel:
$$
(\mathcal{C}f)(q)=\lim_{\varepsilon\rightarrow 0}
\int_{\partial\mathcal{U}_n}S(q+\varepsilon\mathbf{e_0},
\omega)f(\omega)d\beta(\omega),\quad \forall f\in
L^2(\partial\mathcal{U}_n),~q\in\partial\mathcal{U}_n.
$$
i.e., $\mathcal{C}(f)$ is the boundary limit of some function in
$\mathcal{H}^2(\mathcal{U}_n),$ where the limit is taken in
$L^2(\partial\mathcal{U}_n)$ norm. $\mathcal{C}$ satisfies
$\mathcal{C}^2=\mathcal{C}=\mathcal{C}^\ast$, where
$\mathcal{C}^\ast$ is the adjoint operator of $\mathcal{C}$.

\vskip 0.2cm \noindent\textbf{Theorem 3.3.} \emph{$\mathcal{C}$ can
be extended to a bounded operator from $L^p(\partial\mathcal{U}_n)$
to $L^p(\partial\mathcal{U}_n),$ $1< p<\infty.$}
\renewcommand\proofname{Proof}
\begin{proof}
Represent $\mathcal{C}$ by the integral on
Lie group $\mathcal {Q}_n,$ i.e., consider the following
operator
$$
(\mathcal{C}f)(h)=\lim_{\varepsilon\rightarrow 0}
\int_{\mathcal{Q}_n}K_\varepsilon(g^{-1}\diamond
h)f(g)d\beta(g),\quad \forall f\in
L^2(\mathcal{Q}_n),~h\in\mathcal{Q}_n,
$$
where $K_\varepsilon(h)=S(h(0)+\varepsilon\mathbf{e_0},0)$,
$h=(q',t_1,t_2,t_3).$ Define the dilation on $\mathcal {Q}_n$ by
$\delta\circ(q',t_1,t_2,t_3)=(\delta
q',\delta^2t_1,\delta^2t_2,\delta^2t_3)$ $(\delta>0),$ and the
length of the element by
$\rho(h)=\max\{|q'|,|t_1|^{\frac{1}{2}},|t_2|^{\frac{1}{2}},
|t_3|^{\frac{1}{2}}\},$ respectively. Then $\mathcal {Q}_n$ becomes
a homogeneous group of dimension $d=4n+6$. Consider the distribution
$$K(h)=\lim_{\varepsilon\rightarrow
0}K_\varepsilon(h)=s(\nu)\big|_{\nu=|q'|^2+\mathbf{e}\cdot t},$$ and
write $q'=(y_1,y_2,\ldots,y_{4n}),$ then one can verify that
\begin{align*}
\big|K(h)\big|&\leq C\rho^{-d}(h),\\ \Big|\frac{\partial}{\partial
y_{i}}K(h)\Big|&\leq C\rho^{-d-1}(h), \quad 1\leq i\leq 4n,\\
\Big|\frac{\partial}{\partial t_{j}}K(h)\Big|&\leq C\rho^{-d-2}(h),
\quad 1\leq j\leq3.
\end{align*}

Because $\mathcal{C}$ is a projection operator, $\mathcal{C}$ must
be of type $(2, 2)$ and
$\|\mathcal{C}f\|_{L^2(\partial\mathcal{U}_n)}\leq
\|f\|_{L^2(\partial\mathcal{U}_n)}.$ Then one can finish the proof
by the theory of harmonic analysis on homogeneous groups (see, e.g.
\cite{Ste}).
\end{proof}

\vskip 0.16cm \noindent\textbf{Corollary 3.4.} \emph{The integral
operator
\begin{eqnarray*}
(Cf)(q)=\int_{\partial\mathcal{U}_n}S(q,\omega)f(\omega)
d\beta(\omega),\quad q\in \mathcal{U}_n
\end{eqnarray*}
is bounded from $L^p(\partial\mathcal{U}_n)$ to $\mathcal
{H}^p(\mathcal{U}_n)$ for $1<p<\infty.$}

\section{Octonionic Heisenberg group and Hardy spaces on the
octonionic Siegel half space}

Denote $\mathbb{B}$ and $\mathcal{U}$ the unit ball and Siegel half
space in $\mathbb{O}^2$ respectively, viz.,
$$
\mathbb{B}=\big\{(\sigma_1,\sigma_2)\in
\mathbb{O}^2:~|\sigma_1|^2+|\sigma_2|^2<1\big\},
$$
$$
\mathcal{U}=\big\{(\tau_1,\tau_2)\in \mathbb{O}^2:~
\mbox{Re}\,\tau_2>|\tau_1|^2\big\}.
$$
Then one can verify that the Cayley transform
\begin{eqnarray*}
\sigma_1=\frac{2\tau_1(1+\overline{\tau_2})}{|1+\tau_2|^2},\quad
\sigma_2=\frac{(1+\overline{\tau_2})(1-\tau_2)}{|1+\tau_2|^2}=
\frac{(1-\tau_2)(1+\overline{\tau_2})}{|1+\tau_2|^2}
\end{eqnarray*}
is a bijection from $\mathcal{U}$ to $\mathbb{B}$ with the inversion
being given by
\begin{eqnarray*}
\tau_1=\frac{\sigma_1(1+\overline{\sigma_2})}{|1+\sigma_2|^2},\quad
\tau_2=\frac{(1+\overline{\sigma_2})(1-\sigma_2)}{|1+\sigma_2|^2}=
\frac{(1-\sigma_2)(1+\overline{\sigma_2})}{|1+\sigma_2|^2}.
\end{eqnarray*}
Unlike the case of several complex variables, the above Cayley
transform is neither left $\mathbb{O}$-analytic not right
$\mathbb{O}$-analytic.

The boundary of $\mathcal{U}$ is denoted by
$\partial\mathcal{U}=\big\{(\tau_1,\tau_2)\in \mathbb{O}^2:~
\mbox{Re}\,\tau_2=|\tau_1|^2\big\}.$ We introduce three mappings on
$\mathcal {U}$ here: dilations, rotations and translations. Let
$\tau=(\tau_1,\tau_2)\in \mathbb{O}^2,$ for every positive number
$\delta,$ define the dilation $\delta\circ\tau$ as follows:
$$
\delta\circ\tau=\delta\circ(\tau_1,\tau_2)=(\delta\tau_1,\delta^2
\tau_2),
$$
it is non-isotropic due to the structure of $\mathcal{U}.$ For each
rotation $\mathcal{R}$ on $\mathbb{O},$ define the rotation
$\mathcal{R}(\tau)$ by
$$
\mathcal{R}(\tau)=\mathcal{R}(\tau_1,\tau_2)=
(\mathcal{R}(\tau_1),\tau_2).
$$
Obviously, both the dilation and rotation give self mappings of
$\mathcal{U}$ that can be extended to mappings on the boundary
$\partial\mathcal{U},$ but in general they are neither left
$\mathbb{O}$-analytic not right $\mathbb{O}$-analytic.

Before we describe the translations on $\mathcal{U},$ we introduce
the octonionic Heisenberg group, denoted by $\mathcal{O}.$ This
group consists of the set
$$
\mathbb{O}\times \mathbb{R}^7=\big\{[\omega,t]=[\omega,t_1, \ldots,
t_7]:~\omega \in \mathbb{O}, t=(t_1,\ldots,t_7) \in
\mathbb{R}^7\big\}
$$
with the multiplication law
\begin{eqnarray}\label{mul}
[\alpha,t]\diamond[\beta,s]=[\alpha+\beta, t_1+s_1+2\mbox{Im}_1\,
(\overline{\alpha}\beta),\ldots,t_7+s_7+2\mbox{Im}_7\,
(\overline{\alpha}\beta)],
\end{eqnarray}
which makes $\mathbb{O}\times \mathbb{R}^7$ into Lie group with the
neutral element $[0,0]$ and the inverse element of $[\omega,t]$
being given by $[\omega,t]^{-1}=[-\omega,-t].$

For each element $[\omega,t] \in \mathcal{O},$ we define the
translation on $\mathcal{U}$:
\begin{eqnarray}\label{5.2}
[\omega,t]:~(\tau_1,\tau_2)\mapsto(\tau_1+\omega,\tau_2+
|\omega|^2+2\overline{\omega}\tau_1+\mathbf{e}\cdot t),
\end{eqnarray}
where $\mathbf{e}=(e_1,\ldots,e_7)$, $\mathbf{e}\cdot t=
\sum_{i=1}^7e_it_i.$ This mapping preserves the function
$r(\tau)=\mbox{Re}\,\tau_2-|\tau_1|^2$, hence it maps
$\mathcal{U}=\{\tau: r(\tau)>0\}$ to itself and preserves the
boundary $\partial\mathcal{U}=\{\tau: r(\tau)=0\}.$

Besides, the reader can check that the mapping (\ref{5.2}) defines
an action of the group $\mathcal{O}$ on $\mathcal {U}$. Via this
action at the origin
$$
[\omega,t]:\,(0,0)\mapsto(\omega,|\omega|^2+\mathbf{e}\cdot t),
$$
we can identify $\mathcal{O}$ with $\partial\mathcal {U}:$ $
\mathcal{O}\ni[\omega,t]\mapsto(\omega,|\omega|^2+\mathbf{e}\cdot
t)\in \partial\mathcal{U}.$

Let $dh$ be the Haar measure on $\mathcal {O},$ using the
identification of $\partial\mathcal {U}$ with $\mathcal {O}$ we
introduce the measure $d\beta$ on $\partial\mathcal {U}$ by the
following integral identity:
$$
\int_{\partial\mathcal{U}}
F(\tau)d\beta(\tau)=\int_{\mathbb{O}\times \mathbb{R}^7}
F(\tau_1,|\tau_1|^2+\mathbf{e}\cdot t)d\tau_1dt,\quad \forall
F(\tau)\in C_c(\partial\mathcal{U}),
$$
where $C_c(\partial\mathcal{U})$ is the set of continuous functions
with compact support on $\partial\mathcal{U}$. With this measure we
can define the space $L^p(\mathcal {O})=L^p(\partial\mathcal
{U})~(0<p<\infty)$.

For any function $F(\tau)$ defined on $\mathcal {U},$
$F_\varepsilon(\tau)=F(\tau+\varepsilon\mathbf{e_0})$ is called the
vertical translate of $F(\tau),$ here
$\mathbf{e_0}=(0,0,0,0,0,0,0,0,1,0,0,0,0,0,0,0).$ Obviously, if
$\varepsilon>0,$ then $F_\varepsilon(\tau)$ is well defined in some
neighborhood of $\partial\mathcal{U}.$ In particular,
$F_\varepsilon(\tau)$ is well defined on $\partial\mathcal{U}.$ For
$\frac{6}{7}<p<\infty,$ we define the Hardy space
$\mathcal{H}^p(\mathcal{U})$ on the octonionic Siegel half space to
be the set of functions $F(\tau)$ which are left
$\mathbb{O}$-analytic on $\mathcal{U}$ with respect to $\tau_1,$
$\tau_2$ respectively, and satisfy
$$
\|F\|_{\mathcal{H}^p(\mathcal{U})}:=\Big(\sup_{\varepsilon>0}\int
_{\partial\mathcal{U}}
|F_\varepsilon(\tau)|^pd\beta(\tau)\Big)^{1/p}<\infty.
$$

$\mathcal{H}^p(\mathcal{U})$ is a Banach space when $1\leq
p<\infty,$ a Fr\'{e}chet space when $\frac{6}{7}<p<1.$ In
particular, $\mathcal{H}^2(\mathcal{U})$ becomes a real-linear
Hilbert space with the inner product $\langle\cdot, \cdot\rangle$
being defined by (it will be seen later that the definition is
reasonable):
$$\langle F, G\rangle=\lim_{\varepsilon\rightarrow
0}\int_{\partial\mathcal{U}}
\overline{G_\varepsilon(\tau)}F_\varepsilon(\tau)d\beta(\tau),
\quad\forall F, G\in \mathcal {H}^2(\mathcal {U}).$$

Our main result concerning $\mathcal{H}^p(\mathcal{U})$ is:

\begin{theo}\label{o1}
Suppose $F(\tau)\in\mathcal{H}^p(\mathcal{U})$
$(\frac{6}{7}<p<\infty),$ then
\begin{enumerate}[(i)]
\item  There exists an $F^b\in
L^p(\partial\mathcal{U}),$ such that
$F_\varepsilon(\tau)|_{\partial\mathcal{U}}\rightarrow F^b$
$(\varepsilon\rightarrow 0)$ in the sense of
$L^p(\partial\mathcal{U})$ norm and almost everywhere.

\item $\{F^b\}$ is closed subspace of $L^p(\partial\mathcal {U}).$
Moreover,

\item $\|F^b\|_{L^p(\partial\mathcal
{U})}=\|F\|_{\mathcal{H}^p(\mathcal{U})}.$
\end{enumerate}
\end{theo}
\vskip 0.3cm Its proof is based on the Subharmonicity of powers of
$\mathbb{O}$-analytic functions and some lemmas concerning the Hardy
space $\mathcal{H}^p(\mathbb{R}_+^8).$

Denote the upper half space $\{\tau=u+\mathbf{e}\cdot v: u>0\}
\subset \mathbb{R}^8$ by $\mathbb{R}_+^8.$ The Hardy space
$\mathcal{H}^p(\mathbb{R}_+^8)$ ($\frac{6}{7}<p<\infty$) is defined
to be the set of functions $f(\tau)$ which are left
$\mathbb{O}$-analytic on $\mathbb{R}_+^8$ with respect to
$\tau=u+\mathbf{e}\cdot v$ and satisfy
$\|f\|_{\mathcal{H}^p(\mathbb{R}_+^8)}:=\big
(\sup_{u>0}\int_{\mathbb{R}^7}|f(u+\mathbf{e}\cdot
v)|^pdv\big)^\frac{1}{p}<\infty.$

\vskip 0.3cm \noindent \textbf{Lemma 4.2.}\emph{ \cite{AD} Suppose
$f\in C^1(\Omega,\mathbb{O})$ satisfies $Df=0,$  then $|f|^p$ is
subharmonic in $\Omega$ when $p\geq \frac{6}{7}.$} \vskip 0.2cm By
analogous discussions as in \cite{DH,GM}, one can prove that \vskip
0.3cm\noindent \textbf{Lemma 4.3.}\emph{ For any
$f\in\mathcal{H}^p(\mathbb{R}_+^8)~(\frac{6}{7}<p<\infty)$, there
holds
\begin{eqnarray}\label{5.6}
\int_{\mathbb{R}^7}\sup_{u>0}|f(u+\mathbf{e}\cdot v)|^pdv\leq C_p
\|f\|_{\mathcal{H}^p(\mathbb{R}_+^8)}^p.
\end{eqnarray}
And there exists an $f^b\in L^p(\mathbb{R}^7)$, such that
\begin{equation}\label{5.7}
f(u+\mathbf{e}\cdot v)\rightarrow f^b(v)~(u\rightarrow 0),\quad
\mbox{for a.e.}\,v\in\mathbb{R}^7,
\end{equation}
\begin{equation}\label{5.8}
\lim_{u\rightarrow 0}\int_{\mathbb{R}^7}|f(u+\mathbf{e}\cdot
v)-f^b(v)|^pdv=0.
\end{equation}
Moreover,
\begin{equation}\label{5.9}
\|f^b\|_{L^p(\mathbb{R}^7)}=\|f\|_{\mathcal{H}^p(\mathbb{R}_+^8)}.
\end{equation}
} \vskip 0.3cm\noindent \textbf{Lemma 4.4.}\emph{ For each $F\in
\mathcal{H}^p(\mathcal{U})$ $(\frac{6}{7}<p<\infty)$, $\tau_1\in
\mathbb{O}$ and $\delta>0,$ we have
$$
f(\tau_2):=F(\tau_1,\tau_2+\delta+|\tau_1|^2)\in\mathcal{H}^p(\mathbb{R}_+^8).
$$
}

\noindent\textbf{Remark.} For $\mathcal {H}^2(\mathcal {U}^n)$ with
several complex variables or $\mathcal {H}^p(\mathcal {U}_n)$ with
several quaternionic variables, the proof of the lemma corresponding
to Lemma 4.4 is to consider the special case when $\tau_1=0$ and
$\delta>0$ first, then prove the general case through the action of
Heisenberg group on the Siegel half space. However, this method is
not valid for Lemma 4.4, since a left (right) $\mathbb{O}$-analytic
function in general does not preserve the analyticity after left
(right) multiplying the variable by an octonion constant. In fact,
we have

\vskip0.3cm\noindent\textbf{Proposition 4.5.} \emph{Suppose
$F(q_1,q_2)$ is left $\mathbb{H}$-regular on $\mathbb{H}\times
\mathbb{H}$ with respect to $q_1$ and $q_2$ respectively, then for
any quaternionic constant $\alpha,$ $f(q_1):=F(q_1,\alpha q_1)$ is
left $\mathbb{H}$-regular on $\mathbb{H}$ with respect to $q_1.$}

\vskip0.3cm\noindent\textbf{Proposition 4.6.} \emph{Suppose
$F(\tau_1,\tau_2)$ is left $\mathbb{O}$-analytic on
$\mathbb{O}\times \mathbb{O}$ with respect to $\tau_1$ and $\tau_2$
respectively, then the following two conditions are equivalent:
\begin{enumerate}[(i)]
\item For any octonionic constant $\alpha,$
$f(\tau_1):=F(\tau_1,\alpha\tau_1)$ is left $\mathbb{O}$-analytic on
$\mathbb{O}$ with respect to $\tau_1.$
\item For each fixed $\tau_1\in \mathbb{O},$ all components of
$g(\tau_2):=\overline{F(\tau_1,\tau_2)}$ consist a Stein--Weiss
conjugate harmonic system on $\mathbb{R}^8.$
\end{enumerate}
} \vskip0.2cm Proposition 4.5 can be verified by direct computation,
Proposition 4.6 is a corollary of the following:
\vskip0.3cm\noindent\textbf{Proposition 4.7.}
\emph{$\forall\alpha\in \mathbb{O}$, $f(\alpha x)$ is left
$\mathbb{O}$-analytic on $\mathbb{O}$ $\Longleftrightarrow$
$\forall\alpha\in \mathbb{O}$, $f(x \alpha)$ is right
$\mathbb{O}$-analytic on $\mathbb{O}$ $\Longleftrightarrow$ All
components of $\overline{f(x)}$ consist a Stein--Weiss conjugate
harmonic system on $\mathbb{R}^8$}
\begin{proof}
Denote $f(x)=\sum_{j=0}^7f_je_j$, $\alpha=\sum_{l=0}^7\alpha_le_l$,
$y(x)=\alpha x=\sum_{k=0}^7y_k(x)e_k$, then $y_k=\mbox{Re}((\alpha
x)\overline{e_k})$, $\frac{\partial y_k}{\partial
x_i}=\mbox{Re}((\alpha e_i)\overline{e_k})=\sum_{l=0}^7\alpha_l
\mbox{Re}((e_le_i)\overline{e_k})$. Thus we have
\begin{align*}
D(f(\alpha x))&=\sum_{i,j}e_ie_j\frac{\partial f_j}{\partial
x_i}\\&=\sum_{i,j,k}e_ie_j\frac{\partial f_j}{\partial
y_k}\frac{\partial y_k}{\partial x_i}
\\&=
\sum_{i,j,k,l}\alpha_le_ie_j\frac{\partial f_j}{\partial
y_k}\mbox{Re}((e_le_i)\overline{e_k})
\\&=
\sum_{j,k,l}\alpha_l(\varepsilon \overline{e_l}e_k)e_j\frac{\partial
f_j}{\partial y_k}\mbox{Re} ((e_l(\varepsilon
\overline{e_l}e_k))\overline{e_k})\qquad (\varepsilon=1 \mbox{ or }
-1)
\\&=
\sum_{j,k,l}\varepsilon^2 \alpha_l
(\overline{e_l}e_k)e_j\frac{\partial f_j}{\partial y_k}
\\&=
\sum_{j,k}(\overline{\alpha}e_k)e_j\frac{\partial f_j}{\partial y_k}
\\&=
\sum_{j,k}\overline{\alpha}(e_ke_j)\frac{\partial f_j}{\partial
y_k}+\sum_{j,k}[\overline{\alpha},e_k,e_j]\frac{\partial
f_j}{\partial y_k}
\\&=\overline{\alpha}Df(y)+\sum_{1\leq j\neq k\leq 7}
[\overline{\alpha},e_k,e_j]
\frac{\partial f_j}{\partial y_k}
\\&=\overline{\alpha}Df(y)+\sum_{1\leq j<k\leq 7}
(\frac{\partial f_j}{\partial y_k}
-\frac{\partial f_k}{\partial y_j})
[\overline{\alpha},e_k,e_j],
\end{align*}
from which it is not difficult to get (cf. \cite{LL})
\begin{align*}
&\forall\alpha\in \mathbb{O}, D(f(\alpha x))=0\\
\Longleftrightarrow&\frac{\partial f_0}{\partial
x_0}=\sum_{i=1}^7\frac{\partial f_i}{\partial x_i},~\frac{\partial
f_0}{\partial x_i}=-\frac{\partial f_i}{\partial x_0}~(1\leq i\leq
7)\mbox{ and }\frac{\partial f_j}{\partial x_k}=\frac{\partial
f_k}{\partial x_j}~(1\leq j<k\leq 7).
\end{align*}
\end{proof}

\begin{proof}[\textbf{Proof of Lemma 4.4}] Write $\tau_2=u+\mathbf{e}\cdot v$,
for any $\omega\in \mathbb{O}$, applying Lemma 4.2 to the function
\begin{eqnarray*}
f(u+\mathbf{e}\cdot v)=F(\omega,u+\delta+|\omega|^2+\mathbf{e}\cdot
v),\quad u>0,
\end{eqnarray*}
we get
\begin{eqnarray}\label{5.10}
|f(u+\mathbf{e}\cdot v)|^p\leq
c_{\omega,\delta}\int_{|\tau_1|<\alpha,\,\,
|\tau_2|<\delta/2}|F(\tau_1+\omega,\tau_2+u
+\delta+|\omega|^2+\mathbf{e} \cdot v)|^pd\tau_1d\tau_2,
\end{eqnarray}
where $\alpha>0$ satisfies $\alpha^2+2\alpha|\omega|=\delta/2,$ and
$c_{\omega,\delta}^{-1}$ is the volume of the cuboid
$\{(\tau_1,\tau_2)\in \mathbb{O}^2: |\tau_1|<\alpha,
|\tau_2|<\delta/2\}$. Since $|\mbox{Re}\,\tau_2|<\delta/2$ and
$u>0,$ we have $\mbox{Re}\,(\tau_2+u+\delta+|\omega|^2)\in
(\delta/2+|\omega|^2,3\delta/2+|\omega|^2),$ thus
$\mbox{Re}\,(\tau_2+u+\delta+|\omega|^2)>|\tau_1+\omega|^2,$ which
guarantees that the integration in (\ref{5.10}) makes sense. Write
$\tau_2=x+\mathbf{e}\cdot y$ and integrate (\ref{5.10}) on both
sides with respect to $v$ over $\mathbb{R}^7,$ after applying the
Fubini theorem, we deduce that
\begin{align*}
&\int_{\mathbb{R}^7}|f(u+\mathbf{e}\cdot v)|^pdv\\ \leq&
c_{\omega,\delta}\int_{|\tau_1|<
\alpha}\int_{\mathbb{R}^7}\int_{|\tau_2|<\delta/2}
|F(\tau_1+\omega,\tau_2+u+\delta+|\omega|^2+\mathbf{e}\cdot
v)|^pd\tau_1dvd\tau_2\\ \leq&
c'_{\omega,\delta}\int_{|\tau_1|<\alpha}
\int_{\mathbb{R}^7}\int_{|x|<\delta/2}|
F(\tau_1+\omega,x+u+\delta+|\omega|^2+\mathbf{e}\cdot
y)|^pd\tau_1dydx.
\end{align*}
Making change of variables $x+u+\delta+|\omega|^2=
\varepsilon+|\tau_1+\omega|^2.$ Since $|x|<\delta/2$ and
$|\tau_1+\omega|^2<\delta/2+|\omega|^2,$ the range of the new
variable $\varepsilon$ lies in the interval
$(u,u+3\delta/2+|\omega|^2).$ So the last integral is majorized by
\begin{align*}
&c'_{\omega,\delta}\int_{u}^{u+3\delta/2+|\omega|^2}
\int_{|\tau_1|<\alpha}\int_{\mathbb{R}^7}
|F(\tau_1+\omega,\varepsilon+|\tau_1+\omega|^2+\mathbf{e}\cdot
y)|^pd\tau_1dyd\varepsilon\\ \leq&
c'_{\omega,\delta}(3\delta/2+|\omega|^2)\|F\|_{
\mathcal{H}^p(\mathcal{U})}^p<\infty.
\end{align*}
i.e., $f(u+\mathbf{e}\cdot v)=F(\omega,u+\delta+
|\omega|^2+\mathbf{e}\cdot v)\in \mathcal{H}^p(\mathbb{R}_+^8).$
\end{proof}

\begin{proof}[\textbf{Proof of Theorem \ref{o1}}]
Applying the maximal inequality (\ref{5.6}) and equality (\ref{5.9})
to the function $f(\tau_2)=F(\tau_1,\tau_2+\delta +|\tau_1|^2)$
($\tau_2=x+\mathbf{e}\cdot y$), we obtain
$$
\int_{\mathbb{R}^7}\sup_{x>0}|F(\tau_1,x+\delta+|\tau_1|^2+\mathbf{e}\cdot
y)|^pdy\leq
C_p\int_{\mathbb{R}^7}|F(\tau_1,\delta+|\tau_1|^2+\mathbf{e}\cdot
y)|^pdy,
$$
integrate it over $\tau_1\in \mathbb{O}$ and write $x=\varepsilon,$
one gets
$$
\int_{\partial\mathcal{U}}\sup_{\varepsilon>0}|F(\tau+
(\varepsilon+\delta)\mathbf{e_0})|^pd\beta(\tau)\leq
C_p\int_{\partial\mathcal{U}}|F(\tau+\delta\mathbf{e_0})|^pd\beta(\tau).
$$
Letting $\delta\rightarrow0$ on both sides, by Fatou's lemma we see
that
$$
\int_{\partial\mathcal{U}}\sup_{\varepsilon>0}|F(\tau+
\varepsilon\mathbf{e_0})|^pd\beta(\tau)\leq
C_p\|F\|_{\mathcal{H}^p(\mathcal{U})}^p<\infty.
$$

We conclude from the above inequality that for almost all $\tau_1\in
\mathbb{O},$ the function $F(\tau_1,\tau_2+|\tau_1|^2),$ as a
function of $\tau_2,$ belongs to $\mathcal{H}^p(\mathbb{R}_+^8).$
Thus by (\ref{5.7}) we see that the limit
$\lim_{\varepsilon\rightarrow
0}F(\tau+\varepsilon\mathbf{e_0})=F^b(\tau)$ exists for almost every
$\tau\in\partial\mathcal{U}.$ Note that
\begin{align*}
\int_{\partial\mathcal{U}}|F_\varepsilon(\tau)-
F^b(\tau)|^pd\beta(\tau)\leq2^p
\int_{\partial\mathcal{U}}\sup_{\varepsilon>0}
|F_\varepsilon(\tau)|^pd\beta(\tau)\leq 2^pC_p\|F\|
_{\mathcal{H}^p(\mathcal{U})}^p<\infty,
\end{align*}
by Lebesgue's dominated convergence theorem we arrive at
$$\lim_{\varepsilon\rightarrow
0}\int_{\partial\mathcal{U}}|F_\varepsilon(\tau)-
F^b(\tau)|^pd\beta(\tau)=0.$$

Now we show the property (\mbox{iii}). By Fatou's lemma, it follows
that
\begin{align}\label{5.11}
\int_{\partial\mathcal{U}}|F^b(\tau)|^pd\beta(\tau)\leq\sup_
{\varepsilon>0}
\int_{\partial\mathcal{U}}|F(\tau+\varepsilon\mathbf{e_0})|
^pd\beta(\tau)=\|F\|_{\mathcal{H}^p(\mathcal{U})}^p.
\end{align}
On the other hand, (\ref{5.9}) leads to
\begin{align*}
\int_{\mathbb{R}^7}|F(\tau_1,\varepsilon+|\tau_1|^2+\mathbf{e}\cdot
y)|^pdy&\leq\sup_{\varepsilon>0}\int_{\mathbb{R}^7}|F(\tau_1,\varepsilon+
|\tau_1|^2+\mathbf{e}\cdot
y)|^pdy\\&=\int_{\mathbb{R}^7}|F^b(\tau_1,|\tau_1|^2+\mathbf{e}\cdot
y)|^pdy,
\end{align*}
integrating it over $\tau_1\in \mathbb{O}$ on both sides and taking
supremum over $\varepsilon>0,$ combining (\ref{5.11}) we get
property (\mbox{iii}).

To prove the assertion (\mbox{ii}), one can deduce from the
inequality (\ref{5.10}) that for any compact set $\mathbb{K}\subset
\mathcal {U},$ there always holds
$$
\sup_{\tau\in \mathbb{K}}|F(\tau)|\leq C_{\mathbb{K},
p}\|F\|_{\mathcal{H}^p(\mathcal{U})},
$$
from which we conclude that if a sequence $\{F_n\}$ converges in
$\mathcal{H}^p(\mathcal{U})$ norm (or metric), then it must converge
uniformly on any compact subset of $\mathcal{U},$ which implies that
the space $\mathcal{H}^p(\mathcal{U})$ is complete with respect to
its norm (or metric). The proof of Theorem \ref{o1} is complete .
\end{proof}

\noindent\textbf{Remark.} Similarly, we can generalize Theorem
\ref{o1} to $\mathbb{O}^n.$ However, octonions is non-associative,
up to now, we are still not sure about the existence of the
octonionic Cauchy--Szeg\"{o} kernel, or how to derive the exact form
of the kernel provided that it exists?

\vskip 0.8cm \noindent{\Large\textbf{Acknowledgements}}

\vskip 0.2cm \noindent This work was partially supported by the
Industry-University-Research Institute Program Sponsored by
Guangdong Province and Academy of Science No. 2009B080701077. We
would like to thank Prof. Liu Heping and Peng Lizhong who suggested
us consider the related problems. This work was started in Autumn of
2008, while the first author was studying in South China Normal
University, under the supervision of the second author. Wang Jinxun
is grateful to Prof. Li Xingmin, for his guidance, discussions and
suggestions. The main results in this paper were obtained in Spring
of 2009, and were presented in a talk given by the first author in
International Conference on Harmonic Analysis and Partial
Differential Equations with Applications held in Beijing Normal
University in May of 2009. We thank the organizers of this
conference for their hospitality. Sincere thanks are also due to
Prof. Qian Tao, who invited us to visit University of Macau in
December, 2009, during that time we had the opportunity to
communicate with the group in Macau about the work in this paper.

\end{document}